\documentclass[a4paper,reqno]{amsart}
\def\@logofont{\footnotesize}
\def\@setaddresses{\par
  \nobreak \begingroup
  \footnotesize
  \def\author##1{\nobreak\addvspace\bigskipamount}%
  \def\\{\par\nobreak}%
  \interlinepenalty\@M
  \def\address##1##2{\begingroup
    \par\addvspace\bigskipamount\indent
    \@ifnotempty{##1}{(\ignorespaces##1\unskip) }%
    {\scshape\ignorespaces##2}\par\endgroup}%
  \def\curraddr##1##2{\begingroup
    \@ifnotempty{##2}{\nobreak\indent\curraddrname
      \@ifnotempty{##1}{, \ignorespaces##1\unskip}\/:\space
      ##2\par}\endgroup}%
  \def\email##1##2{\begingroup
    \@ifnotempty{##2}{\nobreak\indent\emailaddrname
      \@ifnotempty{##1}{, \ignorespaces##1\unskip}\/:\space
      \ttfamily##2\par}\endgroup}%
  \def\urladdr##1##2{\begingroup
    \def~{\char`\~}%
    \@ifnotempty{##2}{\nobreak\indent\urladdrname
      \@ifnotempty{##1}{, \ignorespaces##1\unskip}\/:\space
      \ttfamily##2\par}\endgroup}%
  \addresses
  \endgroup
}
\renewcommand*\subjclass[2][2010]{%
  \def\@subjclass{#2}%
  \@ifundefined{subjclassname@#1}{%
    \ClassWarning{\@classname}{Unknown edition (#1) of Mathematics
      Subject Classification; using '2000'.}%
  }{%
    \@xp\let\@xp\subjclassname\csname subjclassname@#1\endcsname
  }%
}

\usepackage{xcolor}
\usepackage{graphicx,amsmath,amssymb,hyperref}
\usepackage{algorithm}
\usepackage{mathdots, comment}
\usepackage[noend]{algpseudocode}

\makeatletter
\def\BState{\State\hskip-\ALG@thistlm}
\makeatother
\allowdisplaybreaks

\theoremstyle{plain}
\newtheorem{theorem}{Theorem}[section]
\newtheorem{proposition}[theorem]{Proposition}
\newtheorem{lemma}[theorem]{Lemma}

\newtheorem{question}[theorem]{Question}

\theoremstyle{definition}
\newtheorem{definition}[theorem]{Definition}
\newtheorem{example}[theorem]{Example}

\newtheorem{remark}[theorem]{Remark}
\def\supp{\mathrm{supp}}

\begin{document}
\title[Matchings in abelian groups and vector subspaces]{Results and questions on  matchings in abelian groups and vector subspaces of fields}

\author[M. Aliabadi]{Mohsen Aliabadi}
\address{Mohsen Aliabadi \\ 
Department of Mathematics\\ Iowa State University}%
\email{aliabadi@iastate.edu}

\author[K. Filom]{Khashayar Filom}
\address{Khashayar Filom\\ 
Department of Mathematics\\ University of Michigan}%
\email{filom@umich.edu}


\begin{abstract}
A matching from a finite subset $A$ of an abelian group to another subset $B$ is a bijection $f:A\rightarrow B$ with the property that $a+f(a)$ never lies in $A$. 
A matching is called acyclic if it is uniquely determined by its multiplicity function. 
Motivated by a question of E. K. Wakeford on canonical forms for symmetric tensors, the study of matchings and acyclic matchings in abelian groups was initiated by  C. K. Fan and J. Losonczy in \cite{MR1371651,MR1618439}, and was later generalized to the context of vector subspaces in a  field extension \cite{MR2735391,MR3393940}.
We discuss the acyclic matching  and weak acyclic matching properties and we provide results on the existence of acyclic matchings in finite cyclic groups. As for field extensions, we completely classify field extensions with the linear acyclic matching property. The analogy between matchings in abelian groups and in field extensions is highlighted throughout the paper and numerous open questions are presented for further inquiry.
\end{abstract}

\maketitle

\section{Introduction}
The notion of matchings in abelian groups was introduced by Fan and Losonczy in \cite{MR1371651} in order to generalize a geometric property of lattices in Euclidean space.   The study of acyclic matchings was motivated by an old problem of Wakeford concerning  finding sets of monomials which are removable from a generic homogeneous polynomial through a linear change of coordinates \cite{MR1576066}. This linear algebra question motivated C. K. Fan and J. Losonczy to define the concept of acyclic matchings in $\mathbb{Z}^n$ in \cite{MR1371651} which was later generalized to abelian groups by the latter author \cite{MR1618439}.
Matchings have been investigated for non-abelian groups as well \cite{MR2388613}, but we solely work with abelian groups. Throughout this paper,  $G$ denotes an additive abelian group. 
\begin{definition}\label{main definition}
Let $B$ be a finite subset of  $G$ which does not contain the neutral element. For any subset $A$ in $G$ with the same cardinality as $B$, a {\it matching} from $A$ to $B$ is defined to be a bijection $f:A\to B$ such that for any $a\in A$ we have $a+f(a)\not\in A$. For any matching $f$ as above, the associated \textit{multiplicity function} $m_f:G\to \mathbb{Z}_{\geq0}$ is defined via the rule:
\begin{equation}\label{main criterion}
\forall x\in G,\quad m_ f(x)=\#\{a\in A\,:\, a+ f(a)=x\}.
\end{equation}
A matching $ f:A\to B$ is called {\it acyclic} if for any matching $g:A\to B$, $m_f=m_g$ implies $f=g$.
\end{definition}
In view of Definition \ref{main definition}, a natural question to ask is whether two  finite subsets $A$ and $B$ of $G$ satisfying  $\#A=\#B$ and $0\notin B$ can be matched or be acyclically matched, i.e. is there a matching or an acyclic matching from $A$ onto $B$?
It is known that there exists a matching $f:A\rightarrow B$ if $A=B$, if every element of $B$ is a generator of $G$, or if $G$ is torsion-free \cite{MR1618439}. The latter result in particular implies that a torsion-free abelian group $G$ possesses the \textit{matching property}: For any two subsets $A$ and $B$ as in Definition \ref{main definition}, there exists a matching $f:A\rightarrow B$. In \cite{MR1618439}, Losonczy proves that abelian groups with the matching property are precisely those that are either torsion-free or cyclic of prime order; namely, groups that do not possess any non-trivial proper finite subgroup. 
Indeed, torsion-free abelian groups admit the stronger \textit{acyclic matching property} in the sense that for any $A$ and $B$ of the same cardinality with $0\notin B$, there exists an acyclic matching $ f:A\to B$ \cite{MR1409421,MR1618439}. The situation for groups $\Bbb{Z}/p\Bbb{Z}$ of prime orders is more subtle. Paper \cite{MR3393940} shows that
for primes $p$ with $p\equiv -1 \pmod{8}$ the group $\Bbb{Z}/p\Bbb{Z}$ does not have the acyclic matching property via exhibiting an explicit subset of $\Bbb{Z}/p\Bbb{Z}$ that does not admit any acyclic matching onto itself. Based on experimental evidence, it is conjectured in \cite{MR3991937} that $\Bbb{Z}/p\Bbb{Z}$ does not admit the acyclic matching property for any prime $p>5$. We shall prove the following theorems on the existence of matchings between certain subsets of a cyclic group of prime order. 
\begin{theorem}\label{main1}
Let $A$ be subsets of the cyclic group $\Bbb{Z}/p\Bbb{Z}$ where $p$ is a prime number. Suppose $A$ satisfies $A\cap 2A=\emptyset$ and is of size $k$ where $k.2^{k-1}<p$.\footnote{For any integer $m$, $mA$ denotes the subset $\{ma\,:\, a\in A\}$.} 
Then $A$ is acyclically matched to itself via the identity map.  
\end{theorem}

\begin{theorem}\label{main2}
Let $p$ be a prime number and suppose $A$ and $B$ are finite subsets of $\Bbb{Z}/p\Bbb{Z}$ with $0\notin B$ which are of the same size k. If 
$k\leq\sqrt{\log_2p}-1$, then there exists an acyclic matching $f:A\rightarrow B$.
\end{theorem}
\noindent The theorems will be established in \S2; the proof of the first one is based on a linear algebra argument while the second one utilizes a result from additive number theory.  

The condition $A\cap 2A=\emptyset$ in Theorem \ref{main1} is necessary for ${\rm{id}}:A\rightarrow A$ to be a matching. In general, all bijections $f:A\rightarrow B$ are matchings provided that $A\cap (A+B)=\emptyset$.\footnote{The sumset $A+B$ is defined as $\{a+b\,:\, a\in A \text{ and }b\in B\}$.} The profusion of matchings then can possibly imply the existence of an acyclic matching from $A$ to $B$. 
\begin{question}\label{main question}
Let $A,B$ be subsets of the cyclic group $\Bbb{Z}/p\Bbb{Z}$ where $p$ is a prime number. Suppose $A$ and $B$ are of the same size $k$. Does the condition $A\cap(A+B)=\emptyset$ guarantee the existence of an acyclic matching $f:A\rightarrow B$? 
\end{question}
\noindent A partial answer will be provided in Proposition \ref{Answer}. 

In view of the discussion above, an abelian group $G$ is said to admit the \textit{weak acyclic matching property} if there exists an acyclic matching between any two subsets $A$ and $B$ of $G$ that have the same cardinality and satisfy $A\cap (A+B)=\emptyset$. Any cyclic group $\Bbb{Z}/n\Bbb{Z}$ of order smaller than $23$ satisfies the weak acyclic matching property, but the existence of infinitely many cyclic groups  $\Bbb{Z}/p\Bbb{Z}$ of prime order with this property is an open question \cite{MR3991937}. 

The investigation of matchings in abelian groups has an enumerative aspect as well. Paper \cite{MR2847271} for instance provides a lower bound for the number of matchings $A\rightarrow B$ under an assumption on $B$. Using a graph-theoretical interpretation of matchings, in \S2.2 we exhibit bounds on the number of matching $A\rightarrow B$ by invoking some classical results from the theory of permanents; see Proposition \ref{inequality proposition}.   

Given a field extension $L/F$, an analogous notion of matching between two $F$-subspaces of $L$ is developed by Eliahou and Lecouvey in \cite{MR2735391}. 
\begin{definition}\label{linear definition}
Let $A$ and $B$ be two $k$-dimensional $F$-subspaces of $L$. An ordered basis $\mathcal{A}=\{a_1,\ldots,a_k\}$ of $A$ is said to be \textit{matched} to an ordered basis  $\mathcal{B}=\{b_1,\ldots,b_k\}$ of $B$ if 
\begin{equation}\label{linear criterion}
a^{-1}_iA\cap B\subseteq \langle b_1,\ldots,b_{i-1},b_{i+1},\ldots,b_k\rangle
\end{equation}
for each $1\leq i\leq k$. We say that $A$ is matched to $B$ (or $A$ is \textit{matchable} to $B$) if every ordered basis $\mathcal{A}$ of $A$ can be matched to an ordered basis $\mathcal{B}$ of $B$.
\end{definition}
\noindent
To see the analogy, notice that if \eqref{linear criterion} is satisfied, then no $a_ib_i$ can lie in $B=\langle\mathcal{B}\rangle$ and thus, in the multiplicative group $L^\times$, $a_i\mapsto b_i$ defines a matching
$\mathcal{A}\rightarrow\mathcal{B}$  in the sense of Definition \ref{main definition}. One can easily check that having \eqref{linear criterion} for all $i\in\{1,\dots,k\}$ implies 
\begin{equation}\label{intersection criterion}
\dim_F\bigcap_{i\in J}\left(a_i^{-1}A\cap B\right)\leq k-\# J
\end{equation}
for any $J\subseteq\{1,\dots,k\}$. In particular, setting $J=\{1,\dots,k\}$, the subspace 
$\bigcap_{i=1}^k\left(a_i^{-1}A\cap B\right)$
must be trivial which cannot happen if $1\in B$. This brings us to the linear analogue of the matching property in groups.
\begin{definition}\label{linear matching}
A field extension $L/F$ has the \textit{linear matching property} if  every finite-dimensional $F$-subspace $A$ is matched to any other subspace $B$ of $L$ which is of the same dimension and satisfies $1\notin B$.  
\end{definition}
\noindent
Similar to the result from \cite{MR1618439} mentioned above, an extension $L/F$ has the matching property if there is no finite intermediate extension $E/F$ with $E\neq F,L$ \cite{MR2735391}.\footnote{There is a slight gap in the statement of \cite[Theorem 2.6]{MR2735391} which is corrected in \cite{2012arXiv1208.2792E}. The classification of  field extensions with the linear matching property that we mentioned is based on 
\cite[Theorem 2.6]{2012arXiv1208.2792E}.}
Also in the group-theoretic context, we mentioned that if every element of $B$ is a generator of $G$, then there exists a matching $A\rightarrow B$ \cite[Proposition 3.4]{MR1618439}. A similar result has been established in the linear setting: Given a finite field extension $L/F$, two $F$-subspaces $A$  and $B$ of the same dimension are matchable if $B$ is a \textit{primitive} $F$-subspace of $L$ \cite[Theorem 4.2]{MR3723776}. Recall that $B$ is called primitive if $F(\alpha)=L$ for each $\alpha\in B\setminus \{0\}$. We shall show the following regarding the largest possible dimension of a primitive subspace.  
\begin{theorem}\label{primitive}
Let $L/F$ be a finite simple field extension. Then the largest possible dimension of a primitive $F$-subspace of $L$ is given by
$$
[L:F]-\max_{\stackrel{F\subseteq E\subsetneq L}{E \text{ a proper intermediate subfield}}}[E:F].
$$
\end{theorem}
\noindent
A proof will appear in \S3 after a review of linear matchings. Indeed, motivated by \cite{MR2735391}, paper \cite{MR3393940} develops a notion of when $A$ can be \textit{acyclically} matched to $B$, hence a definition of the \textit{linear acyclic matching property} for field extensions. The former is relevant only when $A\cap AB=\{0\}$, compare with condition $A\cap(A+B)=\emptyset$ in 
Question \ref{main question} from the group-theoretic setting.   
Unlike the case of abelian groups, in the linear setting, the linear matching property for field extensions is equivalent to the acyclic one. In \S3, after reviewing the definition of the linear acyclic matching property, we shall prove the following by building on the arguments appeared in \cite{MR3393940}:
\begin{theorem}\label{linear theorem}
A field extension $L/F$ admits the linear acyclic matching property if and only if there is no finite intermediate extension $E/F$ with $E\neq F,L$.
\end{theorem}
\noindent
This generalizes \cite[Theorem 4.5]{MR3393940}. 

\subsection*{Outline}
We have devoted \S2 to matchings in the context of abelian groups and \S3 to linear matchings in the context of field extensions. In \S2.1, after a brief review  of the literature on matchings and acyclic matchings, we prove Theorem \ref{main1} and Theorem \ref{main2}, as well as Proposition \ref{Answer} that provides a partial answer to Question \ref{main question}. These results are all concerned with the existence of acyclic matchings. Some questions on counting the number of matchings are discussed in \S2.2. Proof of Theorem \ref{primitive} on primitive subspaces in a simple field extension appears in \S3.1. Finally, in \S3.2, 
we prove Theorem \ref{linear theorem} that characterizes field extensions with the linear acyclic matching property.

\section{Matchings in abelian groups}
\subsection{Acyclic matchings}
We begin with two examples of matchings and a related definition. 
\begin{definition}
The \textit{support} of a matching $f:A\rightarrow B$, denoted by $\supp(f)$, is the subset of elements $x\in A$ at which the multiplicity function $m_f:G\rightarrow\Bbb{Z}_{\geq 0}$ is positive; that is, the subset of elements $x\in G$ that may be realized as $a+f(a)$ for an element $a$ of $A$. 
\end{definition}

\begin{example}
Let us examine matchings $f:A\rightarrow B$ in which the subsets $A$ and $B$ of the abelian group $G$ are as large as possible. If $\#A=\#B=\#G-1$, then $B=G\setminus\{0\}$, and $f$ should 
be in the form of
$$
f:A=G\setminus\{g_1\}\rightarrow B=G\setminus\{0\}: a\mapsto g_1-a
$$
for an appropriate $g_1\in G$. The support of $f$ only contains $g_1$ at which $m_f$ takes the value $\#G-1$. This clearly shows that the matching above is acyclic.

Next, suppose $A$ and $B$ are of cardinality $\#G-2$. Write $A$ as $G\setminus\{g_1,g_2\}$ and $B$ as $G\setminus\{0,g_3\}$ where $g_1\neq g_2$ and $g_3\neq 0$. 
We claim that any matching $f:A\rightarrow B$ is acyclic. For any $a\in A$, $f(a)$ should be either $g_1-a$ or $g_2-a$ because otherwise $a+f(a)\in A$. Suppose the former occurs $l$ times and the latter $\#G-2-l$ times. 
Therefore, the multiplicity function $m_f:G\rightarrow\Bbb{Z}_{\geq 0}$ attains the value $l$ at $g_1$, the value $\#G-l-2$ at $g_2$, and is zero elsewhere. The knowledge of the subset $A':=\{a\in A\,:\, f(a)=g_1-a\}$ determines the matching. Notice that $(g_2-g_1)+A'\subset A'\cup\{g_1,g_2\}$ because otherwise there exists $a'\in A'$ with $a'':=(g_2-g_1)+a'\in A\setminus A'$ which is impossible since then $f(a')=g_1-a'$ coincides with $f(a'')=g_2-a''$. 
One can easily check that due to the inclusion $(g_2-g_1)+A'\subset A'\cup\{g_1,g_2\}$ and $g_1\notin A'$,  $A'$ should be a progression in the form of 
$\{g_1+(g_1-g_2),g_1+2(g_1-g_2),\dots, g_1+l(g_1-g_2)\}$.
\end{example}

\begin{example}
Let us consider a family of matchings $f:A\rightarrow B$ for which $A\cap (A+B)=\emptyset$ as in Question \ref{main question}. 
Imposing an extra condition that $B\cup\{0\}$ is a subgroup of $G$, we claim that any matching $f:A\rightarrow B$ is acyclic: Given $g:A\rightarrow B$ with $m_g=m_f$, any $a+f(a)$ should be realizable as $a'+g(a')$ too 
($a,a'\in A$). But then $a=a'+(g(a')-f(a))$ where $g(a')-f(a)\in B\cup\{0\}$. This contradicts $A\cap(A+B)=\emptyset$ unless $g(a')=f(a)$ which also implies $a=a'$.
\end{example}

We next turn into results concerned with the existence of matchings or acyclic matchings. As mentioned in the introduction, it is established in \cite{MR1618439} that an abelian group has the matching property if and only if it is torsion-free or cyclic of prime order. The proof therein utilizes Hall's marriage theorem and a result of Kneser on the size of sumsets in abelian groups. 

\begin{theorem}[\cite{MR1618439}]\label{t*}
Let $G$ be an abelian group which is either torsion-free or cyclic of prime order. Suppose $A$ and $B$ are two finite subsets of $G$ which are of the same size and $0\notin B$. Then there exists a matching $f:A\rightarrow B$. 
\end{theorem}

The preceding theorem indicates that an abelian group which is either torsion-free or of  prime order has the matching property. As for the stronger acyclic matching property, torsion-free abelian groups admit the latter property as established in \cite{MR1618439} whereas there exist infinitely many primes $p$ for which $\Bbb{Z}/p\Bbb{Z}$ does not possess the acyclic matching property. 

\begin{theorem}[\cite{MR3393940}]\label{Jafari}
There are infinitely many primes $p$ for which $\Bbb{Z}/p\Bbb{Z}$ does not satisfy the acyclic matching property. 
\end{theorem}
\noindent
Indeed, it is proved in \cite{MR3393940} that if $p\equiv -1 \pmod 8$ or if the multiplicative order  of $2$ modulo $p$ is odd\footnote{The subset of odd primes $p$ for which the multiplicative order ${\rm{ord}}_p2$ is odd is of density $\frac{7}{24}$ according to a result of Hasse \cite{MR205975}.} then the acyclic matching property fails for $\Bbb{Z}/p\Bbb{Z}$. The basic idea is that if a matching $f:A\rightarrow A$ is acyclic, then $f$ and $f^{-1}$ coincide since obviously $m_f=m_{f^{-1}}$. But then, assuming that $\#A$ is odd, $f$ must have a fixed point $a\in A$ which requires  $a+f(a)=2a$ to be outside $A$. But one may choose a subset $A\subset(\Bbb{Z}/p\Bbb{Z})^\times$ with odd cardinality which is invariant under multiplication by $2$, e.g. the set of quadratic residues when $p\equiv -1 \pmod 8$, or the multiplicative subgroup generated by $\overline{2}$ when ${\rm{ord}}_p2$ is odd.

We next turn into Question \ref{main question} regarding the existence of acyclic matching from $A$ onto $B$ whenever all bijections $A\rightarrow B$ are matchings. Proposition below establishes this under the assumption that $B$, following the terminology of \cite{MR810691}, is a \textit{Sidon set}.   
\begin{proposition}\label{Answer}
Let $A,B$ be subsets of an abelian group $G$. Suppose $A$ and $B$ are of the same size satisfying $A\cap (A+B)=\emptyset$. Then there exists an acyclic matching $f:A\rightarrow B$ if we assume the equation $x+y=z+w$ has no solution in $B$ with 
$\{x,y\}\cap\{z,w\}=\emptyset$.
\end{proposition}
\begin{proof}
Aiming for a contradiction, let $k$ be the smallest cardinality for which the proposition is false. Label elements of $A$ and $B$ as $A=\{a_1,\dots,a_k\}$ and $B=\{b_1,\dots,b_k\}$. Pick indices $i,j\in\{1,\dots,k\}$ arbitrarily. 
Bijections $f:A\rightarrow B$ with 
$a_i\mapsto b_j$ are all matchings and are in a one-to-one correspondence with the bijections 
$A\setminus\{a_i\}\rightarrow B\setminus\{b_j\}$. No such $f$ is  acyclic, so there exists another bijection $g:A\rightarrow B$ with the same multiplicity function. It is possible to find such a $g$ with $g(a_i)\neq b_j$ as otherwise no bijection (matching) $A\setminus\{a_i\}\rightarrow B\setminus\{b_j\}$ would be acyclic contradicting the minimality of $k$. Thus there exist $f,g:A\rightarrow B$
with the same multiplicity functions that satisfy $f(a_i)=b_j$ and $g(a_i)\neq b_j$. In particular, 
$a_i+f(a_i)=a_i+b_j$ should be in the support of $g$ as well; that element may be realized as $a_{i'}+g(a_{i'})=a_{i'}+b_{j'}$ for suitable 
$i',j'\in\{1,\dots,k\}$. Notice that $i\neq i'$ and $j\neq j'$ as otherwise $a_i+b_j=a_{i'}+b_{j'}$ implies $g(a_i)=b_j$.  
We deduce that: For any $i,j\in\{1,\dots,k\}$, there exist $i',j'\in\{1,\dots,k\}$ with 
$a_i-a_{i'}=b_{j'}-b_j$ where $i\neq i'$ and $j\neq j'$.
Fixing $i\in\{1,\dots,k\}$ and letting $j$ vary, by the pigeonhole principle, there exist $i'\in\{1,\dots,k\}\setminus\{i\}$ which is associated with two different indices $j$ and $\tilde{j}$: $a_i+b_j=a_{i'}+b_{j'}$ and $a_i+b_{\tilde{j}}=a_{i'}+b_{\tilde{j'}}$ where 
$i'\neq i$, $j'\neq j$, ${\tilde{j'}}\neq{\tilde{j}}$ and $j\neq\tilde{j}$. These equations may be written as 
$b_{j'}-b_j=b_{\tilde{j'}}-b_{\tilde{j}}=a_i-a_{i'}\neq 0$. Therefore, $b_j+b_{\tilde{j'}}=b_{j'}+b_{\tilde{j}}$
where no element from the left appears on the right.
\end{proof}

\begin{example}
Consider the cyclic group $\Bbb{Z}/n\Bbb{Z}$, and let $k$ be a positive integer with $k>1$ and $(k-1)(2^{k-1}+1)<n$. Take $B$ to be the geometric progression
$\left\{\overline{1},\overline{2},\dots,\overline{2^{k-1}}\right\}$.
The reader can easily check that the condition of Proposition \ref{Answer} on $B$ is satisfied. Hence any subset $A$ of $\Bbb{Z}/n\Bbb{Z}$ of size $k$ with 
$A\cap(A+B)=\emptyset$ admits an acyclic matching onto $B$. One can for instance take $A$ to be an arithmetic progression such as 
$\left\{\overline{a},\overline{a}+\overline{2^{k-1}+1},\dots,\overline{a}+\overline{(k-1)(2^{k-1}+1)}\right\}.$
\end{example}

In the case of \textit{symmetric matchings} where $A=B$, Theorem \ref{main1} establishes the existence of an acyclic matching in the absence of the Sidon condition imposed in Proposition \ref{Answer}, but at the expense of limiting the cardinality of $A$.
\begin{proof}[Proof of Theorem \ref{main1}]
Aiming for a contradiction, take $A$ to be a subset with the smallest possible size for which the identity map ${\rm{id}}:A\rightarrow A$ -- which is a matching due to $A\cap 2A=\emptyset$ -- is not acyclic. 
Therefore, writing $A$ as $\{a_1,\dots,a_k\}$ where $k:=\#A$, there should be a bijection $a_i\mapsto a_{\sigma(i)}$ with $\sigma\in{\rm{S}}_k\setminus\{{\rm{id}}\}$ which is a matching and has the same multiplicity function. In other words, $\{2a_i\}_{1\leq i\leq k}$ and 
$\{a_i+a_{\sigma(i)}\}_{1\leq i\leq k}$ coincide as multi-sets. Hence there must exist a second permutation $\lambda\in{\rm{S}}_k$ which yields the multi-set $\{2a_i\}_{1\leq i\leq k}$ as a re-ordering of the multi-set 
$\{a_i+a_{\sigma(i)}\}_{1\leq i\leq k}$, i.e. 
\begin{equation}\label{identities}
a_{\lambda(i)}+a_{\sigma(\lambda(i))}=2a_i \text{ for all } i\in\{1,\dots,k\}.
\end{equation}
The equation above and $\sigma\neq{\rm{id}}$ clearly imply that both permutations $\lambda$ and $\sigma\circ\lambda$ are different from the identity.
Moreover, no proper non-empty subset of $\{1,\dots,k\}$ is preserved by both $\sigma$ and $\sigma\circ\lambda$ (and hence $\lambda$) because for a proper subset of $A$, the identity map should provide an acyclic self-matching.  
Equation \eqref{identities} can be rephrased in terms of permutation matrices: Denoting the permutation matrix corresponding to a permutation $\nu\in{\rm{S}}_k$ with 
${\rm{P}}_\nu=[p_{ij}:=\delta_{i\nu(i)}]_{1\leq i,j\leq k}$,
the vector $[a_1\dots a_k]^{\rm{T}}$ -- whose entries are distinct -- should lie in the null space of $2{\rm{I}}_k-{\rm{P}}_\lambda-{\rm{P}}_{\sigma\circ\lambda}$. 
We claim that the hypothesis $k.2^{k-1}<p$ of the theorem implies that if permutations $\alpha,\beta\in{\rm{S}}_k\setminus\{{\rm{id}}\}$ generate a transitive subgroup of ${\rm{S}}_k$, then  no vector of $\Bbb{F}_p^k$ with distinct entries lies in the null space of $2{\rm{I}}_k-{\rm{P}}_\alpha-{\rm{P}}_\beta$. The resulting contradiction will then conclude the proof.

Notice that the vector  $[1\dots 1]^{\rm{T}}$ is in the kernel of $2{\rm{I}}_k-{\rm{P}}_\alpha-{\rm{P}}_\beta$. Our approach is to show that, over $\Bbb{F}_p$, the nullity of this matrix is one.  The key idea is to first prove this in characteristic zero: The matrix is real; so it suffices to show that the entries of any real vector in its kernel must be identical. This easily follows from the fact that $\Bbb{R}$ is ordered: if the entries of a real vector $[x_1\dots x_k]^{\rm{T}}$ satisfy $2x_i=x_{\alpha(i)}+x_{\beta(i)}$ for all $i$, then $\left\{i\in\{1,\dots,k\}\,:\, x_i\geq x_j\, \forall j\right\}$ is invariant under both $\alpha$ and $\beta$, and thus must coincide with $\{1,\dots,k\}$.
Now the characteristic polynomial $q(t)\in\Bbb{Z}[t]$ of $2{\rm{I}}_k-{\rm{P}}_\alpha-{\rm{P}}_\beta$ is monic with constant term $0$.
We claim that its coefficient of $t$ is non-zero; namely, the algebraic multiplicity of $0$ as an eigenvalue of $2{\rm{I}}_k-{\rm{P}}_\alpha-{\rm{P}}_\beta$ is one as well. This is due to the fact that the decomposition 
$$
\Bbb{R}^k=\{x_1+\dots+x_k=0\}\oplus\Bbb{R}.\{(1,\dots,1)\}
$$
is invariant under the transformation $2{\rm{I}}_k-{\rm{P}}_\alpha-{\rm{P}}_\beta$: The first subspace contains its image and the second one, as established, is its kernel. Hence $q(t)$ is $t$ times the characteristic polynomial for the restriction of $2{\rm{I}}_k-{\rm{P}}_\alpha-{\rm{P}}_\beta$ to its invariant subspace $\{x_1+\dots+x_k=0\}$. The latter polynomial has a non-zero constant term since $\{x_1+\dots+x_k=0\}$ intersects the kernel $\Bbb{R}.\{(1,\dots,1)\}$ of 
$2{\rm{I}}_k-{\rm{P}}_\alpha-{\rm{P}}_\beta$ trivially. We conclude that the coefficient of $t$ in the characteristic polynomial $q(t)$ of $2{\rm{I}}_k-{\rm{P}}_\alpha-{\rm{P}}_\beta$ is non-zero. 
If this remains true modulo $p$, then the rank of $2{\rm{I}}_k-{\rm{P}}_\alpha-{\rm{P}}_\beta$ over $\Bbb{F}_p$ will remain $k-1$. This is going to be achieved by arguing that the absolute value of the non-zero coefficient of $t$ in $q(t)$ should be less than $p$ if $k.2^{k-1}<p$. This coefficient is $(-1)^{k-1}$ times the sum of all $(k-1)\times(k-1)$ minors of $2{\rm{I}}_k-{\rm{P}}_\alpha-{\rm{P}}_\beta$ along the diagonal.
In each column of $2{\rm{I}}_k-{\rm{P}}_\alpha-{\rm{P}}_\beta$, and thus in each column of these minors, the sum of positive entries, as well as the sum of opposites of negative entries, is at most two. 
Thus the minors cannot be larger than $2^{k-1}$ due to an inequality on determinants of real matrices from  \cite{MR485920}.
This results in the desired bound $k.2^{k-1}$ for the absolute value of the coefficient of $t$ in $q(t)$.
\end{proof}

\begin{remark}
In the proof above, one can replace the finite field $\Bbb{F}_p$ with $\Bbb{F}_{p^n}$: The rank of $2{\rm{I}}_k-{\rm{P}}_\alpha-{\rm{P}}_\beta$ will remain $k-1$ over any extension $\Bbb{F}_{p^n}$ of $\Bbb{F}_p$. Consequently, Theorem \ref{main1} remains valid with the additive group $(\Bbb{Z}/p\Bbb{Z})^n$ of $\Bbb{F}_{p^n}$ in place of $\Bbb{Z}/p\Bbb{Z}$. 
\end{remark}

Theorem \ref{main2} provides a result similar to Theorem \ref{main1} on the existence of acyclic matchings, but now we deal with general matchings $f:A\rightarrow B$ and the condition $A\cap(A+B)=\emptyset$ is dropped at the expense of making subsets smaller than what appears in Theorem \ref{main1}.  The existence of acyclic matchings was established in \cite[Theorem 1]{MR1409421} by Alon et al. in the case of subsets of $\Bbb{Z}^n$. The proof uses the existence of a total ordering on $\Bbb{Z}^n$ in an essential way. In \cite[Theorem 4.1]{MR1618439}, Losonczy generalizes this result to torsion-free abelian groups by observing that any torsion-free abelian group admits a total ordering (cf. \cite{MR0007779}). Below, we prove Theorem \ref{main2} by invoking a theorem from  arithmetic combinatorics  that allows one to order the elements of a small enough subset of $\Bbb{Z}/p\Bbb{Z}$ in a certain way compatible with the group structure.

\begin{proof}[Proof of Theorem \ref{main2}]
We reproduce the proof of \cite{MR1618439} by utilizing a \textit{rectification principle} which asserts that a sufficiently small subset of $\Bbb{Z}/p\Bbb{Z}$ may be embedded in integers while preserving certain additive properties. We shall use the sharpest possible version established in \cite{MR2452847}: 
\begin{itemize}
\item For any subset $X$ of $\Bbb{Z}/p\Bbb{Z}$ with $\#X\leq\log_2 p$ there exists an injection $\varphi:X\hookrightarrow\Bbb{Z}$ with the property that a relation such as $x+y=z+w$ among the elements of $X$ implies 
\begin{equation}\label{rectification}
\varphi(x)+\varphi(y)=\varphi(z)+\varphi(w).    
\end{equation}
\end{itemize}
This result is applicable to $X:=(A+B)\cup A\cup B\cup\{0\}$ since its size is no larger than
$$
\#(A+B)+\#A+\#B+1\leq (\# A)(\# B)+\#A+\#B+1=(\#A+1)^2\leq\log_2 p.
$$
Replacing $\varphi$ with $\varphi-\varphi(0)$, one may assume that $\varphi$ sends the identity element of $G$ to zero. We can then call an element $b$ of $B$ positive or negative according to the sign of the integer $\varphi(b)$. In particular, if $b\in B$ is positive, then for every $a\in A$, $\varphi(a+b)$ -- which is equal to $\varphi(a)+\varphi(b)$ by \eqref{rectification} -- is larger than $\varphi(a)$:  
$$
\varphi(a+b)=\varphi(a+b)+\varphi(0)=\varphi(a)+\varphi(b)>\varphi(a).
$$
Using this fact, we  construct an acyclic matching from $A$ to $B$ first in the case that elements of $B$ are all positive. Write elements of $A$ as $a_1,a_2,\dots,a_k$ so that 
\begin{equation}\label{ordering}
\varphi(a_1)<\varphi(a_2)<\dots<\varphi(a_k).
\end{equation}
Starting from $a_1$, notice that $a_1+B\not\subseteq A$ due to $0\notin B$. Thus there exists $b\in B$ with $a_1+b\notin A$. Define $f(a_1)$ to be the element $b$ of $B$ with this property for which $\varphi(b)$ is as small as possible. 
One can continue in this manner inductively: Suppose the values of $f$ at $a_1,\dots,a_{i-1}$ are defined. We have 
\begin{equation}\label{construction}
a_i+(B\setminus\{f(a_1),\dots,f(a_{i-1})\})\not\subseteq A\setminus\{a_1,\dots,a_{i-1}\}
\end{equation}
because the subsets have the same cardinality, and the first one does not contain $a_i$ while the second one does. We then define $f(a_i)$ to be an element $b$ from $B\setminus\{f(a_1),\dots,f(a_{i-1})\}$ with
$a_i+b\notin A\setminus\{a_1,\dots,a_{i-1}\}$ and with $\varphi(b)$ as small as possible. Notice that the matching property is satisfied: $a_i+f(a_i)$ does not belong to $A$ since otherwise it should lie in 
$\{a_1,\dots,a_{i-1}\}$ which is impossible because of the positivity of $f(a_i)\in B$: 
$$
\varphi(a_i+f(a_i))>\varphi(a_i)>\varphi(a_1),\dots,\varphi(a_{i-1}).
$$
This procedure results in a bijection $f:A\rightarrow B$ which is a matching. We next show that it is acyclic. 
Assume the contrary: let $g:A\rightarrow B$ be a different matching with the same multiplicity function, i.e. $m_f=m_g$. Since $f\neq g$, one can pick an $x\in A+B$ satisfying 
\begin{equation}\label{auxiliary9}
\{a\in A\,:\, a+f(a)=x\}\neq\{a\in A\,:\, a+g(a)=x\}
\end{equation}
and $\varphi(x)$ as small as possible. The sets from \eqref{auxiliary9} are of the same size since $m_f(x)=m_g(x)$. We can choose an element from the second one which is not in the first: Let $i\in\{1,\dots,k\}$ be the smallest index satisfying $a_i+f(a_i)\neq x$ and $a_i+g(a_i)=x$. We now reach a contradiction: One cannot have $\varphi(a_i+f(a_i))<\varphi(a_i+g(a_i))=\varphi(x)$ because then  \eqref{auxiliary9} holds with $a_i+f(a_i)$ in place of $x$, contradicting the way $x$ was chosen. Hence $\varphi(a_i+f(a_i))\geq \varphi(a_i+g(a_i))$ or, invoking the additivity property, $\varphi(f(a_i))\geq\varphi(g(a_i))$. But $\varphi$ is injective, so 
$\varphi(f(a_i))>\varphi(g(a_i))$. Due to our choice of $i$, the matchings $f$ and $g$ coincide on $\{a_1,\dots,a_{i-1}\}$. So $g(a_i)$ is a member of the subset 
$$B\setminus\{f(a_1),\dots,f(a_{i-1})\}=B\setminus\{g(a_1),\dots,g(a_{i-1})\}$$ 
appearing in \eqref{construction}; and $a_i+g(a_i)\notin A$ since $g$ is a matching. But, in view of $\varphi(f(a_i))>\varphi(g(a_i))$, this violates the way  $f(a_i)$ was chosen.

Finally, we should address the situation where $B$ has negative elements (recall that $0\notin B$). Partition $B$ as $B_-\sqcup B_+$ where 
$$
B_-:=\{b\in B\,:\, \varphi(b)<0\}, \quad B_+:=\{b\in B\,:\, \varphi(b)>0\}.
$$
Denote the size of $B_-$ by $1\leq l\leq k=\#B$. Writing elements of $A$ as $a_1,\dots,a_k$ as before (see \eqref{ordering}), we can similarly partition $A$ into subsets 
$$
A_-:=\{a_1,\dots,a_l\}, \quad A_+:=\{a_{l+1},\dots,a_k\};
$$
which are of the same sizes. From what we have established so far, there is an acyclic matching $A_+\rightarrow B_+$, and also an acyclic matching $A_-\rightarrow B_-$ by a straightforward  modification of our construction above for the case that elements of the target set are all negative. These two acyclic matchings define a matching $f:A=A_-\sqcup A_+\rightarrow B=B_-\sqcup B_+$ which we claim is acyclic as well. 
It suffices to show that any other matching $g:A\rightarrow B$ with $m_f=m_g$ maps $A_-$ onto $B_-$ and $A_+$ onto $B_+$. This is due to the fact that the minimum of the expression
$$
\min_{A'\subset A, \#A'=l}\sum_{a\in A'}(\varphi(a')+\varphi(h(a')))=\min_{A'\subset A, \#A'=l}\sum_{a\in A'}\varphi(a'+h(a'))
$$
as $h$ varies among bijections $A\rightarrow B$ is attained precisely  when $A'=A_-$ and $h(A_-)=B_-$; conditions that $f$ satisfies. The integer above for $h=g$ is the same as the corresponding number when $h=f$ due to $m_f=m_g$. We deduce that $g(A_-)=B_-$ and this concludes the proof.
\end{proof}

\begin{remark}
By invoking a version of the rectification principle for sets with a ``small doubling'' property from \cite[Theorem 2.1]{MR1608875}, in the theorem above one can forgo the logarithmic bound for a linear one: Again, set $X\subset\Bbb{Z}/p\Bbb{Z}$ to be $(A+B)\cup A\cup B\cup\{0\}$. Suppose we have $\#(X+X)<\sigma(\#X)$ where $\sigma>0$. Then there exists a constant $c=c(\sigma)>0$ such that one can construct an acyclic matching $f:A\rightarrow B$ provided that $\#X\leq cp$.
\end{remark}

Theorems \ref{main1} and \ref{main2} raise the following natural question: 
\begin{question}\label{Mersenne}
What is the largest $\epsilon>0$ for which there exists $c_1>0$ and $c_2$ with the property that for any prime number $p$, any two subsets $A$ and $B$ of $\Bbb{Z}/p\Bbb{Z}$  with 
$0\notin B$ and $\#A=\#B\leq c_1(\log_2p)^\epsilon+c_2$ can be matched acyclically?
\end{question}
\noindent
Theorem \ref{main2} implies that the answer to Question \ref{Mersenne} should satisfy $\epsilon\geq\frac{1}{2}$. On the other hand, we do not expect any $\epsilon>1$ to work because there are conjecturally infinitely many Mersenne primes $p$; and for any such prime, the construction appeared in the proof of Theorem \ref{Jafari} exhibits a subset of order $O(\log_2p)$ in $\Bbb{Z}/p\Bbb{Z}$ which admits no acyclic matching onto itself.    

\subsection{Enumerative questions}

Let $G$ be an arbitrary abelian group and suppose $A$ and $B$ are two finite subsets of $G$ of size $k$ with $0\notin B$. The goal of this section is to provide bounds for the number of matchings $A\rightarrow B$
(i.e. $\#\mathcal{M}(A,B)$) in terms of $k$. The key idea is to interpret elements of $\mathcal{M}(A,B)$ as perfect matchings in a certain bipartite graph.

\begin{definition}
Notations as above, the bipartite graph $\mathcal{G}_{A,B}$ associated with $A$ and $B$ has the disjoint union $A\dot\cup B$ as its set of vertices with $a\in A$ connected to $b\in B$ if and only if $a+b\notin A$. 
\end{definition}
\noindent
Matchings $A\rightarrow B$ are clearly in correspondence with the perfect matchings in $\mathcal{G}_{A,B}$ (cf. Definition \ref{main definition}).

\begin{question}
Is there a graph-theoretical interpretation of acyclic matchings from $A$ to $B$?  
\end{question}

In view of the preceding discussion, enumerating matchings $A\rightarrow B$ amounts to counting perfect matchings in a bipartite graph. There is an extensive literature on the problem of counting the number of matchings in a graph \cite{MR2889590, MR2438582}. In particular, it is well-known that the number of perfect matchings in a simple undirected graph $\mathcal{G}$ is not larger than the square root of the \textit{permanent} of its adjacency matrix with equality if $\mathcal{G}$ is bipartite \cite{MR136399,kasteleyn1961statistics}.   
In the case of the bipartite graph $\mathcal{G}_{A,B}$ whose vertices are partitioned into two parts $A$ and $B$  of the same size, the permanent of the adjacency matrix is the square of that of the \textit{biadjacency matrix} 
\begin{equation}\label{biadjacency}
M_{A,B}=[m_{ij}]_{1\leq i,j\leq k}, \quad     m_{ij}:=\begin{cases}1&\mathrm{if}\; a_i+b_j\notin A,\\ 0&\mathrm{otherwise},\end{cases}
\end{equation}
where we have denoted elements of $A$ and $B$ by $a_1,\dots,a_k$ and $b_1,\dots,b_k$ respectively. 
This whole discussion results in the following: 
\begin{proposition}\label{number of matchings}
With subsets $A$ and $B$ of $G$ as above, one has 
$$
\#\mathcal{M}(A,B)={\rm{per}}(M_{A,B})
$$
where $M_{A,B}$ is the matrix from \eqref{biadjacency}.
\end{proposition}

We now arrive at the main result of this section which provides  upper and lower bounds on the number of matchings:

\begin{proposition}\label{inequality proposition}
Suppose $A$ and $B$ are subsets of an abelian group $G$ of the same size. For each $a\in A$ and $b\in B$ define:
$$
A_b:=\{a'\in A\,:\, a'+b\notin A\},\quad B_a:=\{b'\in B\,:\, a+b'\notin A\}.
$$
The number of matchings from $A$ to $B$ admits the upper bound below:
\begin{equation}\label{inequality}
\#\mathcal{M}(A,B)\leq\min\left\{\prod_{a\in A}((\#B_a)!)^{\frac{1}{\#B_a}},\prod_{b\in B}((\#A_b)!)^{\frac{1}{\#A_b}}\right\}.\footnote{Here $0!^{\frac{1}{0}}$ should be interpreted as zero; if one of the sets $A_b$ or $B_a$ is empty (e.g. when $b=0$), then there is no matching $A\rightarrow B$.}
\end{equation}
Moreover, denoting the size of $A$ and $B$ by $k$, suppose the numbers $\{\#B_a\}_{a\in A}$ and $\{\#A_b\}_{b\in B}$ are written in the increasing order as  
$$
\#B_{a_1}\leq\dots\leq \#B_{a_k}, \quad \#A_{b_1}\leq\dots\leq \#A_{b_k}.
$$
Then we have the following lower bound for the number of matchings from $A$ to $B$:
\begin{equation}\label{inequality'}
\#\mathcal{M}(A,B)\geq\max\left\{\prod_{i=1}^k\max(\#B_{a_i}-i+1,0),\prod_{i=1}^k\max(\#A_{b_i}-i+1,0)\right\}.
\end{equation}

\end{proposition}

\begin{proof}
According to Proposition \ref{number of matchings}, one just needs to establish these inequalities for ${\rm{per}}(M_{A,B})$. Inequality \eqref{inequality} is a result of applying the famous \textit{Bregman-Minc inequality} (\cite{MR155843,bregman1973some}) to the biadjacency matrix $M_{A,B}$. (A similar idea is used in \cite{MR2398830} for bounding the number of perfect matchings in a graph.) 
Inequality \eqref{inequality'} immediately follows from a lower bound for the permanent of $(0,1)$-matrices established on  \cite[p. 26]{MR195879}.
\end{proof}

The final example of this section discusses a  lower bound for the number of symmetric matchings $A\rightarrow A$. 

\begin{example}
Take $A$ to be a subset of an abelian group $G$ that does not contain the identity element and is of size $k$. Suppose all intersections $A\cap (A-a)$ (where $a\in A$) are of the same size $k-r$ where $1\leq r\leq k$. Applying the celebrated \textit{van der Waerden conjecture} (established independently in \cite{egorycev1981solution} and \cite{falikman1981proof}) to the biadjacency matrix of $\mathcal{G}_{A,A}$, we deduce that 
$$
\#\mathcal{M}(A,A)\geq r^k\frac{k!}{k^k}.
$$
One example of subsets $A$ with such an intersection property is the following: Let $\psi:G\rightarrow G$ be a group automorphism and take $A$ to be the orbit of a non-identity element under the action of $\psi$. The intersections $A\cap (A-a)$ are of the same cardinality as $\psi$ bijects them onto each other:
$$
\psi(A\cap (A-a))=\psi(A)\cap (\psi(A)-\psi(a))=A\cap (A-\psi(a)).
$$
\end{example}

\section{Matchings in linear subspaces of field extensions}
\subsection{Primitive subspaces}
Let $L/F$ be a field extension and suppose $A$ and $B$ are two $F$-subspaces of $L$ with $\dim_FA=\dim_FB<\infty$ and $1\notin B$. Primitive subspaces of $L$ (defined in \cite{MR3723776}) naturally arise in deciding if  $A$ is matched to  $B$ (cf. Definition \ref{linear definition}). To elaborate, we review the following situations from the literature where the answer is positive:
\begin{itemize}
\item $A$ is matched to $B$ if the adjunction of any non-zero element of $B$ to $F$ generates $L$, i.e. if the subspace $B$ is primitive \cite[Theorem 4.2]{MR3723776};
\item $A$ is matched to $B$ if $A=B$ \cite[Theorem 5.1]{MR2735391};
\item $A$ is matched to $B$ if $L/F$ has no proper finite intermediate extension $E/F$ of degree larger than one \cite[Theorem 5.2]{MR2735391}.\footnote{A refinement of this result appears in 
\cite[Theorem 5.5]{2012arXiv1208.2792E}: In any extension $L/F$, $A$ can be matched to $B$ if $\dim_FA=\dim_FB$ is smaller than $\min_{F\subsetneq E\subseteq F}[E:F]$. See \cite[Corollary 3.6]{MR3723776}
for the corresponding group-theoretic result.} 
\end{itemize}
Notice that, as discussed in \S1, these results have parallels in the context of matching a finite subset $A$ of an abelian group $G$ to another finite subset $B$ which is of the same size and does not contain the identity element of $G$ (see \cite[Proposition 3.4]{MR1618439}, \cite[Theorem 2.1]{MR1618439} and \cite[Theorem 3.1]{MR1618439} respectively). The proofs utilize a \textit{dimension criterion} which is based on a linear version of Hall's marriage theorem. (Similarly, Hall's marriage theorem is used in the proof of the aforementioned results  from \cite{MR1618439}.)
The dimension criterion asserts that inequalities \eqref{intersection criterion} are necessary and sufficient conditions for an ordered basis $\{a_1,\dots,a_n\}$ of $A$ to be matched to an ordered basis of $B$.

We now focus on proving Theorem \ref{primitive}. Given a finite extension $L/F$, primitive subspaces are those $F$-subspaces of $L$ that intersect any extension $E\subsetneq L$ of $F$ only trivially. By the primitive element theorem, there are only finitely many intermediate subfields if and only if $L/F$ is a simple extension (of course there is no primitive subspace unless $L/F$ is simple).  Therefore, to determine the largest possible dimension of a primitive subspace of $L$ in the setting of Theorem \ref{primitive}, one needs to determine the same for $F$-subspaces which intersect members of a certain finite family $\mathcal{V}$ of $F$-subspaces of $L$ trivially -- $\mathcal{V}$ being the family of proper intermediate subfields of the extension $L/F$. This is easier to do if the base field $F$ is infinite or at least large enough; see Lemma \ref{infinite field} below. However, finite-dimensional vector spaces over finite fields may be covered by finitely many of their proper subspaces. So, in order to establish Theorem \ref{primitive} for finite $F$, one should take into account that $\mathcal{V}$ here is a special family of subspaces whose members are subfields. This is more subtle and will be discussed in Lemma \ref{finite field}. 
\begin{example}
Let $E:=F(a)$ be a proper subfield of $L$ where $k:=[E:F]>1$. Then  $A:=E$ is not matched to any subspace of the form $B:=\langle a,\dots,a^{k-1},x\rangle$ where $x\in L\setminus E$ (\cite{MR2735391}). See \cite[Theorem 3.1]{MR1618439} for the group-theoretic analogue.
\end{example}

\begin{lemma}\label{infinite field}
Let $V$ be a finite-dimensional vector space over a field $F$ and let $\mathcal{V}=\{V_i\}_{i=1}^{m}$ be a finite family of subspaces of $V$ where $m\leq\#F$. Then the largest possible dimension of a subspace $W$ of $V$ which intersects every member of $\mathcal{V}$  trivially is given by 
$$
\dim_FV-\max_{1\leq i\leq m}\dim_FV_i.
$$
\end{lemma}
\begin{proof}
This is straightforward and follows from the fact that  $V$ cannot be covered by a finite number of its proper subspaces unless the number of subspaces is larger than $\#F$ (in which case $F$ is finite) \cite[Lemma 2]{MR111739}.
\end{proof}

We now turn into finite fields. As usual, for any prime power $q$, the finite field with $q$ elements is denoted by $\Bbb{F}_q$.
\begin{lemma}\label{finite field}
The codimension of the largest $\Bbb{F}_q$-subspace of $\Bbb{F}_{q^n}$ whose all non-zero elements are primitive is the same as the largest possible degree over $\Bbb{F}_q$ that a proper intermediate subfield can attain. 
\end{lemma}

\begin{proof}
Let $p_1<\dots<p_s$ be the prime factors of $n$. The maximal subfields of the extension $\Bbb{F}_{q^n}/\Bbb{F}_q$
are $\Bbb{F}_{q^{n/p_1}},\dots,\Bbb{F}_{q^{n/p_s}}$. The first one has the largest possible degree over $\Bbb{F}_q$ which is $\frac{n}{p_1}$. The goal is to come up with an $\Bbb{F}_q$-subspace $W$ of $\Bbb{F}_{q^n}$ whose dimension is $n-\frac{n}{p_1}$ and intersects each of $\Bbb{F}_{q^{n/p_1}},\dots,\Bbb{F}_{q^{n/p_s}}$ trivially. 

First, notice that by replacing $\Bbb{F}_q$ with the intersection $\bigcap_{i=1}^s\Bbb{F}_{q^{n/p_i}}=\Bbb{F}_{q^{n/p_1\dots p_s}}$ of maximal subfields we can assume that $n$ is a product of primes, say $n=p_1\dots p_s$ where $p_1<\dots<p_s$ as before. The second step is to apply the normal basis theorem: There is an element $\theta\in \Bbb{F}_{q^n}$ for which 
\begin{equation}\label{basis}
\{\sigma^j(\theta)\}_{j=1}^{n=p_1\dots p_s}
\end{equation}
is a basis for $\Bbb{F}_{q^n}$ as a vector space over $\Bbb{F}_q$. Here, $\sigma:x\mapsto x^q$ is the Frobenius element, the generator of 
\begin{equation}\label{Galois}
{\rm{Gal}}(\Bbb{F}_{q^n}/\Bbb{F}_q)\cong\Bbb{Z}/n\Bbb{Z}\cong \Bbb{Z}/p_1\Bbb{Z}\times\dots\times\Bbb{Z}/p_s\Bbb{Z}. 
\end{equation} 
Our strategy is to construct $W$ as the subspace spanned by a subset $\{\sigma^j(\theta)\}_{j\in T}$ of the basis in \eqref{basis}  where $T$ is an appropriate subset of $\Bbb{Z}/n\Bbb{Z}$ of size $n-\frac{n}{p_1}$. Besides the cardinality, since we want all intersections $W\cap\Bbb{F}_{q^{n/p_i}}$ to be trivial, we need the following: For any non-zero vector $(c_j)_{j\in T}$ of elements of $\Bbb{F}_q$, the element 
$\sum_{j\in T}c_j\sigma^j(\theta)$ should not belong to any $\Bbb{F}_{q^{n/p_i}}$. But in the Galois correspondence, the latter field corresponds to the subgroup $\langle\sigma^{n/p_i}\rangle\cong\Bbb{Z}/p_i\Bbb{Z}$ of \eqref{Galois}.  
Hence 
\begin{equation}\label{auxiliary6}
\sum_{j\in T}c_j\sigma^j(\theta)\in \Bbb{F}_{q^{n/p_i}}\Leftrightarrow
\sigma^{n/p_i}\left(\sum_{j\in T}c_j\sigma^j(\theta)\right)=\sum_{j\in T}c_j\sigma^j(\theta).    
\end{equation}
But 
$\sigma^{n/p_i}\left(\sum_{j\in T}c_j\sigma^j(\theta)\right)=\sum_{j\in T}c_j\sigma^{j+\frac{n}{p_i}}(\theta)
=\sum_{j\in T+\frac{n}{p_i}}c_{j-\frac{n}{p_i}}\sigma^{j}(\theta)$, 
where the indices $j$ are considered modulo $n=p_1\dots p_s$ (recall that $T\subset\Bbb{Z}/n\Bbb{Z}$). As the elements of the Galois orbit of $\theta$ are linearly independent (i.e. \eqref{basis} is a basis), equating the coefficients in the identity from \eqref{auxiliary6} implies that $c_j=c_{j-\frac{n}{p_i}}$ for any $j\in T$. But clearly $c_k=0$ for $k\notin T$. We deduce that if $c_j\neq 0$ (there exists such a $j$ as otherwise $\sum_{j\in T}c_j\sigma^j(\theta)=0$), then $c_{j-\frac{n}{p_i}}\neq 0$ and thus 
$j-\frac{n}{p_i}\in T$. Continuing this procedure with $j-\frac{n}{p_i}$ in place of $j$, we observe that if the non-zero element $\sum_{j\in T}c_j\sigma^j(\theta)$ of $W=\langle\{\sigma^j(\theta)\}_{j\in T}\rangle$ lies in $\Bbb{F}_{q^{n/p_i}}$, then $T\subset\Bbb{Z}/n\Bbb{Z}$ must contain an (mod $n$) arithmetic progression of the form 
$$
j, j-\frac{n}{p_i}, j-2\frac{n}{p_i}, \dots, j-(p_i-1)\frac{n}{p_i}, j-p_i\frac{n}{p_i}\stackrel{n}{\equiv} j.
$$      
This boils everything down to the additive nature of
$$T\subset\Bbb{Z}/n\Bbb{Z}\cong\Bbb{Z}/p_1\Bbb{Z}\times\dots\times\Bbb{Z}/p_s\Bbb{Z}.$$
To finish the proof, one needs to construct a subset $T$ of $\Bbb{Z}/p_1\Bbb{Z}\times\dots\times\Bbb{Z}/p_s\Bbb{Z}$
of size $n-\frac{n}{p_1}$ with the following property: For each $1\leq i\leq s$, $T$ should not contain any subset of the form 
\begin{equation}\label{auxiliary7}
\{j_1\}\times\dots\times\{j_{i-1}\}\times\Bbb{Z}/p_i\Bbb{Z}\times\{j_{i+1}\}\times\dots\times\{j_s\}
\end{equation}  
where $j_1,\dots,j_{i-1},j_{i+1},\dots, j_s$ are arbitrary integers considered modulo suitable primes. Instead of $T$, we exhibit $T$ through its complement $T^c$, a subset of $\Bbb{Z}/p_1\Bbb{Z}\times\dots\times\Bbb{Z}/p_s\Bbb{Z}$ of size 
$\frac{n}{p_1}=p_2\dots p_s$ which intersects all subsets of the form \eqref{auxiliary7}.  Pick arbitrary surjections 
$$f_2:\Bbb{Z}/p_2\Bbb{Z}\rightarrow\Bbb{Z}/p_1\Bbb{Z},\dots,f_s:\Bbb{Z}/p_s\Bbb{Z}\rightarrow\Bbb{Z}/p_1\Bbb{Z}$$
(recall that $p_2,\dots, p_s$ are larger than $p_1$) and define $T^c$ as
$$
\left\{\left(\sum_{i=2}^sf_i(j_i),j_2,\dots,j_s\right)\,:\, j_2\in\Bbb{Z}/p_2\Bbb{Z},\dots,j_s\in\Bbb{Z}/p_s\Bbb{Z}\right\}.
$$
It is easy to check that this intersects every subset of the form \eqref{auxiliary7}.
\end{proof}

\begin{proof}[Proof of Theorem \ref{primitive}]
Follows from Lemma \ref{infinite field} if $F$ is infinite and from Lemma \ref{finite field} in the case of finite $F$.  
\end{proof}

\subsection{Linear acyclic matchings}
In this final subsection, we shall prove Theorem \ref{linear theorem} after providing a background on \textit{linear acyclic matchings}. Unlike Definition \ref{linear matching} and similar to matching in abelian groups, a linear acyclic matching is indeed a map such as $f:A\rightarrow B$. Here, $A$ and $B$ are vector subspaces of a certain field extension and $f$ is a linear isomorphism.  The definition of linear acyclic matchings, developed in \cite{MR3393940}, builds on the notion of \textit{strong matchings} from \cite{MR2735391}. 
\begin{definition}\label{strong matching}
Let $L/F$ be a field extension and $A$ and $B$ two $F$-subspaces of $L$ which are of the same finite dimension. An $F$-linear isomorphism $f:A\rightarrow B$ is said to be a strong matching if any ordered basis $\mathcal{A}$ of $A$ is matched to the ordered basis $f(\mathcal{A})$ of $B$ in the sense specified in Definition \ref{linear definition}.
\end{definition}
\noindent
It is known that there is a strong matching from $A$ to $B$ if and only if the intersection $A\cap AB$ is trivial, in which case every linear isomorphism between $A$ and $B$ is a strong matching \cite[Theorem 6.3]{MR2735391}. 
In view of the dimension criterion \eqref{intersection criterion}, this is a special situation  because it implies that the subspaces appearing in \eqref{intersection criterion} are all trivial. 

To define linear acyclic matchings, in analogy with Definition \ref{main definition}, one should first make sense of two linear isomorphisms $f,g:A\rightarrow B$ between vector subspaces of a field $L$ having the same ``multiplicity functions''. We want the elements of the multiplicative group $L^\times$ realized as $af(a)$ to be the same as those realized as $ag(a)$. But here $A$ and $B$ are subspaces rather than finite sets. So article  \cite{MR3393940} puts forward the definition below:
\begin{definition}\label{equivalent}
Let $L/F$ be a field extension and $A,B$ be $F$-subspaces of $L$. Two $F$-linear isomorphisms $f,g:A\rightarrow B$ are called to be \textit{equivalent} if there exists a linear automorphism $\phi:A\rightarrow A$
satisfying 
\begin{equation}\label{auxiliary3}
af(a)=\phi(a)g(\phi(a))   
\end{equation}
for every $a\in A$.
\end{definition}
\noindent 
An obvious way of defining an isomorphism $g:A\rightarrow B$ equivalent to a given $f:A\rightarrow B$ is to pick an $r\in F\setminus\{0\}$ and set  $g(a):=\frac{1}{r^2}\,f(a)$ which satisfies \eqref{auxiliary3} if $\phi(a):=ra$. But is there any other way to come up with an isomorphism equivalent to $f$? This brings us to the definition of the linear acyclic matching property from \cite{MR3393940}.  
\begin{definition}\label{linear acyclic}
Let $L/F$ be a field extension, and suppose $A$ and $B$ are $F$-subspaces of $L$ whose dimensions are finite and equal.  A strong matching $f:A\rightarrow B$ is called acyclic if  any other strong matching $g:A\rightarrow B$ equivalent to it is of the form  $cf$ for some $c\in F$. The extension $L/F$ is said to have the linear acyclic matching property if for every pair $A$ and $B$ of $F$-subspaces of $L$ which are of the same finite dimension and satisfy $A\cap AB=\{0\}$, there exists a linear acyclic matching from $A$ to $B$. 
\end{definition}

We next start working towards the proof of Theorem \ref{linear theorem}. Lemma \ref{lemma} below will be used in the subsequent Proposition \ref{Proposition 1} that establishes the ``if'' part of Theorem \ref{linear theorem}. The statements and the proofs of the lemma and the proposition are respectively adapted from \cite[Lemma 4.3]{MR3393940} and \cite[Theorem 4.5]{MR3393940} with slight modifications: The original statements are only concerned with extensions $L/F$ where elements of $L\setminus F$ are transcendental over $F$ -- extensions that \cite{MR3393940,MR2735391} (rather unconventionally) call ``purely transcendental''. We more generally consider extensions that lack non-trivial proper intermediate subfields finite over the base.

\begin{lemma}\label{lemma}
Let $L/F$ be a field  extension without any non-trivial proper finite  intermediate extension $E/F$. Suppose $A$ and $B$ are two $F$-subspaces of $L$ with 
$$0<\dim_FA=\dim_FB<\dim_FL.$$
If two $F$-linear isomorphisms $f,g:A\rightarrow B$ are equivalent via a linear automorphism $\phi:A\rightarrow A$, then either $g=cf$ for a suitable $c\in F\setminus\{0\}$ or  
$g\circ\phi$ is the multiplication map by some $\alpha\in L\setminus\{0\}$ in which case $B=\alpha A$.
\end{lemma}
\begin{proof}
Fix a non-zero element $x$ of $A$. Changing $a$ to $x$ and $a+x$ in \eqref{auxiliary3} yields $xf(x)=\phi(x)g(\phi(x))$ and $(a+x)f(a+x)=\phi(a+x)g(\phi(a+x))$ for any arbitrary $a\in A$. Combining these with \eqref{auxiliary3} and using the additivity of $f$, $g$ and $\phi$, one obtains 
\begin{equation}\label{auxiliary4}
(x\phi(a)-a\phi(x))(xg(\phi(a))-ag(\phi(x)))=0
\end{equation}
for all $a\in A$. (See \cite[Proof of Lemma 4.3]{MR3393940} for the details for this computation.) As $L$ is a field, one of the parentheses in \eqref{auxiliary4} should be zero. 
The set of elements $a$ of $A$ that make each parentheses zero is an $F$-subspace. But $A$ cannot be written as a union of two proper subspaces. 
So either $\phi:A\rightarrow A$ is given by multiplication by $r:=\frac{\phi(x)}{x}$, or the  linear isomorphism $g\circ\phi:A\rightarrow B$ is the multiplication map by 
$\alpha:=\frac{g(\phi(x))}{x}$ which implies $B=\alpha A$. We only need to further analyze the former situation. The element $r$ must lie in $F$: The finite-dimensional $F$-subspace $A$ of $L$ is invariant under multiplication by $r\in L$, hence $r$ satisfies a monic equation with coefficients in $F$ and of degree $\dim_FA<\dim_FL$, cf. \cite[Proposition 2.4]{MR0242802}. But then $F(r)$ is a proper subfield of $L$ which is finite over $F$, thus should be the same as $F$ due to our assumption about the extension $L/F$. Now, in view of the $F$-linearity of $f$, $g$ and $\phi$, plugging $\phi(a)=ra$ in \eqref{auxiliary3} implies $g=cf$ where $c:=\frac{1}{r^2}$. 
\end{proof}

The lemma above will be used in the proof of the proposition below which is a slight generalization of  \cite[Theorem 4.5]{MR3393940}.
\begin{proposition}\label{Proposition 1}
A field extension $L/F$ without non-trivial proper finite intermediate extensions of the form $E/F$ has the linear acyclic matching property.    
\end{proposition}
\begin{proof}
Let $A$ and $B$ be as in Definition \ref{linear acyclic}: two $F$-subspaces of $L$ of the same finite dimension satisfying $A\cap AB=\{0\}$.  The goal is to show the existence of an $F$-linear isomorphism $f:A\rightarrow B$ which is acyclic in the sense any other isomorphism $g:A\rightarrow B$ equivalent to it can be written as $cf$ for an appropriate $c\in F$. There is nothing to prove if $A=B=\{0\}$. Moreover, $A$ and $B$ are proper since $A\cap AB=\{0\}$ implies $1\notin B$. So one can safely assume that 
$$0<\dim_FA=\dim_FB<\dim_FL$$
as in Lemma \ref{lemma}. Pick an arbitrary isomorphism $f:A\rightarrow B$. If it is acyclic, we are done. Otherwise, the lemma implies that $B=\alpha A$ for some $\alpha\in L\setminus\{0\}$. We claim that the $F$-linear isomorphism 
$$
\tilde{f}:A\rightarrow B=\alpha A:a\mapsto\alpha a
$$
given by multiplication by $\alpha$ is acyclic. If not, there exists another isomorphism $\tilde{g}:A\rightarrow B=\alpha A$ which is not in the form of $c\tilde{f}$ for any $c\in F$ 
but is equivalent to $\tilde{f}$ via an automorphism $\phi:A\rightarrow A$ satisfying 
\begin{equation}\label{auxiliary8}
a(\alpha a)=a\tilde{f}(a)=\phi(a)\tilde{g}(\phi(a))    
\end{equation}
for all $a\in A$.
Invoking Lemma \ref{lemma} once again, there exists $\beta\in L\setminus\{0\}$ such that $B$ can also be written as $\beta A$,  and $\tilde{g}\circ\phi$ is the multiplication map by $\beta$. 
Substituting in \eqref{auxiliary8}, we deduce that $\phi$ is the multiplication map by $\beta^{-1}\alpha$. But, repeating the argument used in the proof of Lemma \ref{lemma},  the element $\beta^{-1}\alpha$ must lie in $F$ due to our assumption on $L/F$ because 
$\alpha A=\beta A$ implies that 
$$[F(\beta^{-1}\alpha):F]\leq \dim_FA<\dim_FL$$
(cf. \cite[Proposition 2.4]{MR0242802}). Plugging $\phi(a)=(\beta^{-1}\alpha)a$ in \eqref{auxiliary8}, the $F$-linearity of $\tilde{g}$ yields $\tilde{g}=(\beta^{-1}\alpha)^{-2}\tilde{f}$. This is a contradiction since we assumed that $\tilde{g}\neq c\tilde{f}$ for all $c\in F$.  
\end{proof}

We next turn into the ``only if'' part of Theorem \ref{linear theorem}.
\begin{proposition}\label{Proposition 2}
Let $L/F$ be a field extension admitting an intermediate subfield $F\subsetneq E\subsetneq L$ with $[E:F]<\infty$. Then $L/F$ does not satisfy the linear acyclic matching property. 
\end{proposition}
\begin{proof}
Motivated by Lemma \ref{lemma}, pick an element $\alpha\in L\setminus E$ and set $A$ and $B$ to be $E$ and $\alpha E$ respectively. Then $A$ and $B$ are finite-dimensional $F$-subspaces satisfying 
$A\cap (AB)=E\cap(\alpha E)=\{0\}.$
Hence every $F$-linear isomorphism $f:A\rightarrow B$ is a strong matching according to \cite[Theorem 6.3]{MR2735391}. 
We claim that there always exists another $F$-linear isomorphism $g:A\rightarrow B$ which is equivalent to $f$ but cannot be written as $cf$. 
Define $g$ as 
$g(a):=\frac{1}{\beta}\,f\left(\frac{a}{\beta}\right)$
where $\beta\in E\setminus F$. This clearly is another $F$-linear isomorphism from $A=E$ onto $B=\alpha E$; and is furthermore equivalent to $f$ because the $F$-linear automorphism $\phi(a):=\beta a$ of $A$ satisfies 
$af(a)=\phi(a)g(\phi(a))$ for all $a\in A$. 
But $g$ is not in the form of $cf$ for any $c\in F$. Otherwise: $\frac{1}{\beta}\,f\left(\frac{a}{\beta}\right)=cf(a)$. Since $f$ takes its values in $\alpha E$ and $E$ is a field containing $F$, 
this requires $\beta$ to lie in $F$, a contradiction.
\end{proof}

\begin{proof}[Proof of Theorem \ref{linear theorem}]
Immediately follows from Propositions \ref{Proposition 1} and  \ref{Proposition 2}.
\end{proof}


\section*{Acknowledgment}
We are deeply grateful to Shmuel Friedland for his constant encouragement, generosity and for many insightful conversations.  We are also grateful to Richard Brualdi and Martin Isaacs for motivating conversations. We would like to thank anonymous referees for their useful comments.

\bibliography{bib}
\bibliographystyle{abbrv}

\end{document}